    	\documentclass[11pt]{amsart}

\usepackage{amsrefs}

     \usepackage[all]{xy}

     \usepackage{enumerate}

     \newenvironment{problist}
{\begin{enumerate}[(a)]}
{\end{enumerate}}

     	\usepackage{amssymb}
     \usepackage{amsmath}
    	\usepackage{amsfonts}
	\usepackage{latexsym}   		
	
    \newtheorem{thm}{Theorem}

    \newtheorem{lem}[thm]   {Lemma}
    \newtheorem{cor}[thm]   {Corollary}

    \newtheorem{prop}[thm]  {Proposition}

    \newcommand{\term}[1]   {{\bf #1}\index{#1}}

    \newcommand{\inclds}     {\hookrightarrow}

\newcommand{\?}{\, ?\, }
\newcommand{\st}{\mid}

    \newcommand{\A}        {\mathcal{A}}
    \newcommand{\B}			{\mathcal{B}}
    \newcommand{\C}        {\mathcal{C}}
    
    \newcommand{\F}        {\mathcal{F}}

                \newcommand{\K}        {\mathcal{K}}

\newcommand{\R}  {\mathcal{R}}

\newcommand{\I}{\mathcal{I}}

\newcommand{\U}{\mathcal{U}}

    \newcommand{\ZZ}    {\mathbb{Z}}

    \newcommand{\NN}    {\mathbb{N}}
    \newcommand{\FF}    {\mathbb{F}}
    \newcommand{\QQ}    {\mathbb{Q}}

    \newcommand{\s}     {\Sigma}
    
    \newcommand{\smsh}  {\wedge}
    \newcommand{\om}    {\Omega}
    
    \newcommand{\wdg}   {\vee}

    \newcommand{\of}    {\circ}
    \newcommand{\id}    {\mathrm{id}}

    \newcommand{\twdl}  {\widetilde}

    \newcommand{\sseq}  {\subseteq}

\renewcommand{\C}{\mathcal{C}}

\DeclareMathOperator{\limone}{lim^1}

\DeclareMathOperator{\holim}{holim}

            \DeclareMathOperator{\cl}    {cl}

 \DeclareMathOperator{\im}    {im}

    \DeclareMathOperator{\Hom}   {Hom}
    \DeclareMathOperator{\Ext}   {Ext}

        \DeclareMathOperator{\conn}{conn}

 \DeclareMathOperator{\Ph}    {Ph}
    
        \DeclareMathOperator{\map}    {map}

\newcommand{\mmathand}{\qquad\mbox{and}\qquad}

\begin{document}
	
	



\title{Finite-Dimensional Spaces in Resolving Classes}



\author{Jeffrey Strom}
\address{Department of Mathematics\\
Western Michigan University\\
Kalamazoo, MI\\
49008-5200 USA}
\email{Jeff.Strom@wmich.edu}	




\begin{abstract}
Using the theory of resolving classes, we show that 
if $X$ is a CW complex of finite type such that  
$\map_*(X, S^{2n+1})\sim *$ for all sufficiently large $n$, 
then $\map_*(X, K) \sim *$ for every simply-connected
finite-dimensional CW complex $K$; and under mild 
hypotheses on $\pi_1(X)$, the same conclusion holds
for \textit{all}   finite-dimensional
complexes $K$. 
Since it is comparatively easy to prove the former condition
for $X = B\ZZ/p$ (we give a proof  in 
an appendix), this result can be applied to give a new, 
more elementary 
proof of the Sullivan conjecture.
\end{abstract}

\maketitle
\pagestyle{myheadings}
\markboth{{J. Strom}}
{{A simple 
homotopy-theoretical
proof of the Sullivan conjecture}}


\section*{Introduction}

Haynes 
Miller proved the Sullivan conjecture 
(that 
the space
of pointed
maps from $B\ZZ/p$ to $K$
 is weakly contractible for all finite-dimensional 
CW complexes $K$)
in the seminal paper \cite{MR750716}.
The heart of Miller's proof  is 
  a herculean feat of pure algebra:  he shows
   that the $E_2$-terms of certain
 Bousfield-Kan spectral sequences---involving homological 
 algebra in the nonabelian category of unstable algebras over
 the Steenrod algebra---vanish.
 %
%
 %
Around the same time, using simpler Massey-Peterson techniques
and   ordinary homological algebra of unstable \textit{modules}
over the Steenrod algebra, he showed that 
$\map_*(B\ZZ/p, S^{2n+1}) \sim *$ for all $n\geq 1$
\cite{MR764593}.
Our goal is to show by purely homotopy-theoretical 
methods that this easier result 
implies the full Sullivan conjecture.

\begin{thm}
\label{thm:main}
Let $X$ be a CW complex of finite type.  
Then the following are 
equivalent:
\begin{enumerate}
\item 
$\map_*(X, S^n)\sim *$ for all sufficiently large $n\equiv 1$ mod  $k$
for some $k$
\item
  $\map_*(X, K)\sim *$ for all simply-connected
finite-dimensional CW complexes $K$.
\end{enumerate}
Furthermore, 
if $\pi_1(X)$ has no nontrivial perfect quotients\footnote{such  groups are
sometimes 
described as \term{hypoabelian}}, then there is no need to restrict the
fundamental group of $K$.
\end{thm}

Our proof relies heavily on the theory of \textit{resolving classes}, 
introduced
in the paper \cite{MR2029919}.  After some preliminaries, 
we give a  streamlined 
 and updated account of the basic
theory of resolving classes, which we hope may be useful 
to other researchers.  Once the theory is in place, the proof 
of Theorem \ref{thm:main} is accomplished in two steps,
depending on whether or not $K$ is simply-connected.
In the final section, we briefly discuss some issues 
related to resolving classes and Theorem \ref{thm:main}.
For example, if $\s X \not\sim *$, 
the condition $\map_*(X, K)\sim *$ for all
simply-connected finite-dimensional complexes 
forces there to be  nontrivial maps from $X$
to certain infinite wedges of
simply-connected finite-dimensional complexes.
 We also offer  some
  some interesting problems and questions concerning
the `sphere codes' $\sigma(X) = \{ n \st \map_*(X, S^n) \sim *\}$.

For completeness, we include, in an appendix, a detailed outline
of the proof of the following
theorem, essentially due to Miller \cite{MR764593}.

\begin{thm}
\label{thm:blarg}
Let
 $X$ be a finite-type CW complex
  such that $\twdl H^*(X; \ZZ[{1\over p}]) = 0$.
If 
$H^*(X; \FF_p)$ is reduced and $H^*(X; \FF_p) \otimes J(n)$
is injective for all $n\geq 0$, 
then   $\map_*(X, S^{2n+1}) \sim *$ for all $n\geq 1$.
\end{thm}

Since  $B\ZZ/p$ satisfies the conditions of    
  Theorems  \ref{thm:main} and  \ref{thm:blarg}, together they
imply the Sullivan conjecture.

\medskip

\paragraph{\bf Acknowledgement}
Many thanks are due to John Harper for pointing out that I had the raw 
materials for a proof of the full Sullivan conjecture,  and for 
bringing the paper \cite{MR764593} to my attention.

\section{Preliminaries}

\subsection{Notation for Collections of Spaces}
\label{subsection:prelimcollections}
Since we will use collections of spaces throughout this paper, 
  it will be helpful to set up some basic notation for them.
Our constructions are homotopy-respecting, 
so we tacitly close all of our collections 
under weak homotopy equivalence.   

If $\A$ is a collection of spaces, then  an expression of a space 
$W$ as a \term{finite-type 
wedge} of spaces in $\A$
is a weak equivalence  $W\sim \bigvee_\I A_i$
where $A_i\in \A$ for each $i\in \I$ and  
 for each $n\in \NN$ only finitely many of the spaces 
 $A_i$ are \textit{not} $n$-connected.
 We say that $W$ is a finite-type wedge of spaces in $\A$ 
 if it has such an expression.
For a finite-type wedge $W$ of spaces in $\A$
  that is $(n-1)$-connected but not
 $n$-connected (which we denote $\conn(W) = n-1$), define 
 \[
 s(W) 
 =
 \min
\left\{ 
 k \
  \left| \
 \begin{array}{l}
 \mbox{$W$ has an expression    
 as a finite-type wedge of }
\\
 \mbox{ spaces in $\A$  with 
 all but $k$ summands   $n$-connected} 
 \end{array}
 \right\}
 \right.
 .
 \]
 We use the function $s$  to impose a partial order on the collection
 of finite-type wedges of spaces in $\A$:
we say that 
  $V < W$ if $\conn(W) < \conn(V)$ or if 
$\conn(V)  = \conn(W)$ and
$s(V) < s(W)$.    

For  collections $\A$ and $\mathcal{B}$, we write
 \begin{eqnarray*}
 \A\smsh \mathcal{B}  &=&\{  A\smsh B \st A\in A, B\in \mathcal{B}\}
 \\
 \s \A &=&  \{  \s A \st A\in \A\}
 \\
 \A^\wdg &=& \{ \mbox{all finite-type wedges of spaces in $\A$}\} 
 .
 \end{eqnarray*}
 Thus we say that $\A$ is closed under suspension 
 if $\s\A\sseq\A$, that $\A$ is closed under smash product
 if $\A\smsh\A \sseq \A$, and so on.  
 Note that $(\A^\wdg)^\wdg = \A^\wdg$.  If either $\A$ or 
 $\B$ is closed under suspension, then so is $\A\smsh\B$.

 


\subsection{Cone Length}

\label{subsection:ConeLength}

Let $\A$ be a collection of spaces.
An \term{$\A$-cone decomposition} of \term{length} $n$
for a map $f: X\to Y$ is a homotopy-commutative
diagram
\[
\xymatrix{
{\quad}\ar@{}[d]|{\mbox{$(\mathcal{D})$}}
&&A_0\ar[d]& A_1\ar[d] && A_{n-1}\ar[d]
\\
&&X_0 \ar[r]\ar[d]_\simeq 
& X_1 \ar[r] 
&  \cdots\ar[r] 
&  X_{n-1}\ar[r] 
& X_n\ar[d]^\simeq
\\
&& X\ar[rrrr]^-f &&&& Y
}
\]
in which     $A_k\in \A$ for all $k$ and 
each sequence $A_k \to X_k \to X_{k+1}$
is a cofiber sequence; 
if $f:X\to Y$ is a homotopy equivalence, 
then it has an $\A$-cone decomposition 
\[
\xymatrix{
{\quad}\ar@{}[d]|{\mbox{$(\mathcal{D})$}}
&\qquad & X \ar[ld]_{\id_X} \ar[rd]^f 
\\
&X\ar[rr]^-f && Y
}
\]
 of length zero. 
The 
\term{$\A$-cone length} of $f$ is
\[
 L_\A(f) = \inf\{ \mathrm{length}( \mathcal{D}) \st 
 \mbox{$\mathcal{D}$ is an $\A$-cone
decomposition of $f$} 
\}.
\]
(Thus $L_\A(f) =  \infty$ if $f$ has no $\A$-cone decomposition.)
The \term{$\A$-cone length} of a space $X$ is $\cl_\A(X) = L_\A(*\to X)$.

%

%
%

\subsection{Phantom Maps}

A \term{phantom map} is a pointed map $f:X\to Y$ from a CW complex
$X$ such that the restriction $f|_{X_n}$ of $f$ to the $n$-skeleton
is trivial for each $n$.    
We write $\Ph(X, Y) \sseq [X, Y]$ for the set of 
pointed homotopy classes
of phantom maps from $X$ to $Y$.
See \cite{MR1361910} for an excellent survey on phantom maps.

If $X$ is the homotopy colimit of a telescope diagram
$\cdots \to X_{(n)}\to X_{(n+1)}\to \cdots$, then 
there is a short exact sequence
of pointed sets 
\[
* \to \limone [\s X_{(n)}, Y] 
\longrightarrow
 [X, Y] 
 \longrightarrow
  \lim [ X_{(n)}, Y]\to *, 
\]
and dually, if $Y$ is the homotopy limit of a tower
$\cdots \gets Y_{(n)}\gets Y_{(n+1)}\gets \cdots$, then there
is a short exact sequence
\[
* \to \limone [ X, \om Y_{(n)}] 
\longrightarrow
 [X, Y] 
\longrightarrow
 \lim [ X, Y_{(n)}]\to * .
\]
In the particular case of the expression of a CW complex $X$ as
the homotopy colimit of its skeleta or of a space $Y$ as the
homotopy limit of its Postnikov system, the kernels are the
phantom sets.

%
%
%


We will be interested in showing that certain phantom sets 
$\Ph(X,Y)$ are trivial.
One useful criterion is that if $G$ is a tower of compact Hausdorff
topological groups and continuous homomorphisms, then $\limone G = *$ 
(see \cite[Prop. 4.3]{MR1361910}).  This is used to prove
the following lemma.

\begin{lem}
\label{lem:CompactLimOneVanish}
Let $\cdots \gets Y_{(s)}    \gets Y_{(s+1)} \gets \cdots$ be a tower 
of spaces such that
each homotopy group $\pi_k(Y_{(s)}   )$ is finite.
If $Z$ is of finite type, then 
$
\limone [ Z, \om Y_{(s)}    ] = *
$.
\end{lem}

\begin{proof}
The homotopy sets $[Z_n, \om^j Y_{(s)}   ]$ are  finite, and 
we give them the discrete topology, resulting in  towers of 
compact groups and continuous homomorphisms.
Fixing $s$ and letting $n$ vary, we find that
 ${\lim}_n^1 [ Z_n , \om^2 Y_{(s)}   ] = *$, 
 and hence the exact sequence
 \[
0 \to   {\lim}^1_n [ Z_n, \om^2 Y_{(s)}] 
\longrightarrow
 [Z, \om Y_{(s)}   ] 
\longrightarrow
{ \lim}_n [ Z_n , \om Y_{(s)}   ]\to 1 
\]
(of groups)
reduces to an isomorphism 
$[Z , \om Y_{(s)}   ] \cong 
{ \lim}_n [ Z_n , \om Y_{(s)}   ]$.
Since $[Z , \om Y_{(s)}   ]$
 is an inverse limit of finite discrete spaces,  
it is compact and Hausdorff; and since the structure maps 
$Y_{(s)}    \to Y_{(s-1)}$ induce maps of the towers that define the topology, 
    the induced maps
$ [Z , \om Y_{(s)}   ]\to [Z , \om Y_{(s-1)}]$ are continuous.
Thus
 ${\lim}_s^1 [Z , \om Y_{(s)}   ] = *$.
\end{proof}

The Mittag-Leffler condition is another
useful criterion for the vanishing of $\limone$.
 A
tower of groups $\cdots \gets G_n \gets G_{n+1}\gets \cdots$
is \term{Mittag-Leffler} if for each $n$ 
the images $\im(G_{n+k}\to G_n)$ stabilize for 
large $k$.  That is, if
there is a function $\kappa : \NN\to \NN$
such that 
 $
 \im( G_{n+k} \to G_n) = \im( G_{n+\kappa(n)} \to G_n)\sseq G_n
 $
whenever $k \geq \kappa(n)$.
 
\begin{prop}
Let $\cdots \gets G_n \gets G_{n+1}\gets \cdots$ be a tower of groups.\begin{problist}
\item 
If the tower  is Mittag-Leffler, then $\limone  G_n = *$.
\item
If each $G_n$ is a countable group, then  the converse
holds:  if $\limone G_n = *$, then the tower is Mittag-Leffler 
\cite[Thm. 4.4]{MR1361910}.
\end{problist}
\end{prop}

Importantly, the Mittag-Leffler condition does not refer to the
algebraic structure of the groups $G_n$.
This observation plays a key role in 
 the following   result (cf. \cite[\S 3]{MR1357793}).

\begin{prop}
\label{prop:SameLoop}
Let  $X$ be a CW complex of finite type, and let
$Y_1$ and $Y_2$ be countable CW complexes 
with $\om Y_1 \simeq \om Y_2$.
Then 
$
\Ph( X, Y_1)  = * 
$
if and only if
$
\Ph(X, Y_2) = *
$.
\end{prop}

\begin{proof}
The homotopy equivalence
 $\om Y_1\simeq \om Y_2$ gives 
 levelwise bijections
$
\{ [\s X_n, Y_1]\}
\cong
\{ [ X_n, \om Y_1]\}
\cong 
\{ [ X_n,\om  Y_2]\}
\cong
\{ [\s X_n, Y_2]\}
$
of towers of sets. 
Since $X$ is of finite type and $Y_1, Y_2$ are countable
CW complexes, these towers are towers of countable 
groups.
Now the  triviality of $\Ph(X,Y_1)$ implies that 
  the first tower is Mittag-Leffler; but then all four towers
 must be   Mittag-Leffler, 
and the result follows.
\end{proof} 

We end our account of phantom maps with a criterion for the vanishing 
of phantom maps into countable wedges of spheres.

\begin{prop}
\label{prop:BLA}
If $Z$ is rationally trivial\footnote{i.e., $\twdl H^*(Z;\QQ) = 0$}
 and of finite type, 
then 
\[
\Ph \left(\mbox{$Z, \bigvee_{i=1}^\infty S^{n_i} $} \right) = *  .
\]
\end{prop}

\begin{proof}
 The Hilton-Milnor
 theorem implies that there is a \textit{weak} product 
 of spheres 
 $P = \prod_\alpha S^{m_\alpha}$
 such that 
 $
 \om \left(  \bigvee_1^\infty S^{n_i}  \right) \simeq \om P
 $
 (that is, $P$ is the (homotopy) colimit of the diagram of finite subproducts
 of the categorical product).
 By Proposition \ref{prop:SameLoop}, 
  it suffices to show that $\Ph(Z, P) = *$.  
 
Since the skeleta of $Z$ are compact, 
every map $\s Z_k \to P$ factors 
 through a 
finite subproduct of $P$, so 
  $[\s Z_k , P]$ is a weak product 
 $\prod_\alpha [\s Z_k, S^{m_\alpha}]$. 
Because $Z$ is rationally trivial and of finite type, 
we have $\Ph(Z, S^m) \cong 
\limone  [ \s Z_k, S^{m}] = *$ for each $m$ \cite{MR1361910}.
These are towers of countable groups, so they must all be
Mittag-Leffler.  
Write $\lambda (n,m)$ for the first $k$ for which the images
 \[
\im\left( 
 [(\s Z)_{n+k}, S^m]
\to   [(\s Z)_{n}, S^m]
   \right)
 \]
stabilize.  Since  $\lambda( n,m) = 0$
 for $m > n+1$,    the set $\{ \lambda (n,m)\st m\geq 0\}$ 
 is finite, and we define $\kappa(n)$ to be its maximum.
%
%
Now it is clear that the images 
\[
\im\left( 
\mbox{$
 \prod_\alpha [\s Z_{n+k} , S^{m_\alpha}]
 \to
  \prod_\alpha [\s Z_n, S^{m_\alpha}]
  $}
  \right)
  \]
  are independent of $k$ for $k \geq \kappa(n)$.
  Thus the tower $\{   \prod_\alpha [\s Z_k, S^{n_\alpha}]\}$
  is Mittag-Leffler,  $\Ph( Z, P) = *$,  and  the proof is complete.
  \end{proof}

\section{Resolving Classes}

We are interested in the condition $\map_*(X,Y)\sim *$.
Since this can happen only for path-connected $X$,   
we tacitly assume that $X$ is path-connected;  thus
$\map_*(X,Y) = \map_*(X, Y_\star)$, where $Y_\star$ is the
basepoint component of $Y$, so we may assume that $Y$
is path-connected too, if we like.

\subsection{Basic Definitions}

Let $\mathfrak{X}$ and $\mathfrak{Y}$ both denote the 
classes of all collections of spaces (we intend to use 
 $\mathfrak{X}$
for domains and $\mathfrak{Y}$ for targets).  We define
functions
\[
\Phi: \mathfrak{X} 
\longrightarrow
 \mathfrak {Y}
\mmathand 
\Theta: \mathfrak{Y}
\longrightarrow
 \mathfrak{X}
\]
by the rules 
\begin{eqnarray*}
\Phi(\mathcal{X}) 
&=&
 \{ Y \st \map_*(X, Y) \sim *\ \mbox{for all}\ X\in \mathcal{X}\}
\\
\Theta(\mathcal{Y}) 
&=&
\{ X \st \map_*(X, Y) \sim *\ \mbox{for all}\ Y\in \mathcal{Y}\}.
\end{eqnarray*}
The maps $\Phi$ and $\Theta$ are a \textit{Galois connection} between 
$\mathfrak{X}$ and $\mathfrak{Y}$, and hence establish 
a bijection between $\im(\Theta)$ and $\im(\Phi)$.  

A collection   $\C \in \im(\Theta)$ is a \textit{strongly closed class};
in particular, $\Theta\of \Phi( \{ X\} )$ is the Bousfield class
$\langle X\rangle$
(studied by A. K. Bousfield, E. Dror Farjoun, W. Chach\'olski, 
and others).  
The collections   $\R\in \im(\Phi)$ have received less attention; 
we call them \term{resolving kernels}.\footnote{Thus
$\Phi(\{X\} ) = \im ( P_X)$ and 
$\Theta\of \Phi( \{ X\} ) = \ker(P_X)$, where $P_X$ denote the 
\term{$X$-nullification} functor (see \cite{MR1392221,MR1408539}).}

It follows formally from the definition
 that a strongly closed class is closed under 
 weak equivalence, 
pointed homotopy colimits and extensions by cofibrations; 
dually, resolving kernels are closed under weak equivalence, 
pointed homotopy limits
and extensions by fibrations.

We call a class $\R\in \mathfrak{Y}$ (which we assume to be 
closed under weak equivalence)
a \term{resolving class} if it is closed under
pointed homotopy limits, and a \term{strong resolving class}
if it is also closed under extensions by fibrations.


 \subsection{Closure Properties for Resolving Classes}
 
The power of the theory of 
closed classes  is founded on a few duality-violating
theorems, such as the Zabrodsky lemma
and 
E. Dror Farjoun's theorem 
relating the fiber of an induced map of homotopy colimits
to the `pointwise fibers'.
Similarly,   resolving classes are useful 
tools by virtue of three formally implausible closure properties.
These properties are best expressed in terms of collections of spaces
rather than one space at a time;
throughout this section we write
 $\A$ and $\B$ to denote collections of simply-connected spaces.
 
 The first result concerns the closure of resolving kernels 
 under the formation of wedges.  
 This is a minor extension
 of     \cite[Prop. 7]{MR2029919} 
 using an argument due
 to W. Dwyer.
  
 \begin{thm}
 \label{thm:1}
 Let $\R$ be a resolving kernel.  
 If $\A\smsh \A\sseq \A$  and $\s \A\sseq\R$, 
 then $\s \A^\wdg\sseq \R$.
 \end{thm}

\begin{proof}
If  $W\in \s\A^\wdg$ then we can write 
 $W = \s A \wdg \s B$, where  
 $B\in\A $  and $\s A < W$ (in the partial order defined
 in Section \ref{subsection:prelimcollections}).
The homotopy fiber $W_1$ of the quotient map 
$\s A \wdg \s B \to \s B$
 is easily seen  to be homotopy equivalent 
to 
\[
\mbox{$W_1\simeq
\s A \rtimes \om \s B = 
\s A \smsh ( \om \s B)_+
\simeq 
\s A \smsh 
\left( \bigvee_{k = 0}^\infty B^{\smsh k} \right)$}, 
\]
where the $0$-fold smash product $B^{\smsh 0}$ is $S^0$;
see \cite{MR0281202} for a proof (or \cite{MR2029919}).
Since $\A$ is closed under smash products,
it follows that 
$W_1\in \s\A^\wdg$, and the displayed 
wedge decomposition  demonstrates that
  $W_1 < W$.    Repeating  this process yields a  tower
\[
W \longleftarrow    W_1 \longleftarrow    W_2\gets    \cdots 
\gets
    W_n\longleftarrow    W_{n+1}\gets    \cdots
\]
in which   $W_n\in \s\A^\wdg$ and  $W_{n+1} < W_n$
for each $n$.
It follows that the connectivity of $W_n$ increases
without bound and so $\holim W_n \sim *$.  

To complete the proof, we use the fact that $\R = \Phi( \mathcal{X})$ is a resolving kernel.  Let $X\in \mathcal{X}$ and observe that 
since $\map_*(X, \s B) \sim *$ for all $B\in \A$, 
the induced maps
\[
\map_*(X, W_{n+1})
\longrightarrow \map_*(X, W_{n})
\]
are weak equivalences for all $n$; hence
\[
\map_*(X, W) \sim
\holim_n \map_*( X, W_{n})  
\sim\map_*(X,  \holim_n W_{n})\sim * , 
\]
so $W\in \R$.
\end{proof}

Next we investigate suspension in resolving classes.
 
 \begin{thm}
 \label{thm:2}
 Let $\R$ be a resolving class. 
If  $\s\A^\wdg \sseq \R$ then $\A^\wdg\sseq  \R$.
 \end{thm}
 


 
 
\begin{proof}
To prove the theorem, it suffices to show that  if $X$
is simply-connected and 
$\bigvee_1^k\s X\in \R$ for all $k$, then $X\in \R$.
This   is a 
 consequence, known to Barratt in the 1950s,  of 
the generalization of the Blakers-Massey theorem to $n$-ads of 
simply-connected spaces, proved in \cite{MR0085509,MR0075589}. 

Start with $n$ copies of the inclusion $X\inclds CX$ and build a strongly
cocartesian $n$-cube by repeatedly forming (homotopy)
pushouts.  The result is an $n$-cube with each entry (except $X$)
a wedge of copies of $\s X$.   Remove $X$
from the cube and form the homotopy limit,  $Y_{(n)}\in \R$.
 The $n$-ad excision theorem implies that
 the
natural comparison map $X \to Y_{(n)}$  is an $n$-equivalence.   
These cubes map to one another, 
leading to a morphism of  towers
\[
\xymatrix{
	\cdots	%
                 \ar@{=}[r]^-{  }
		\ar@{}[rd]|{     }
& 
X	%
                 \ar[d]_{  }  
                 \ar@{=}[r]^-{  }
		\ar@{}[rd]|{     }
& 
X	%
                 \ar[d]_{  }  
                 \ar@{=}[r]^-{  }
		\ar@{}[rd]|{     }
& 
	\cdots	%
                 \ar@{=}[r]^-{  }
		\ar@{}[rd]|{     }
& 
X	%
                 \ar[d]_{  }  
                 \ar@{=}[r]^-{  }
		\ar@{}[rd]|{     }
& 
	X	%
                  \ar[d]^{  }
\\ 
	\cdots	%
	         \ar[r]^-{  }
&	       
Y_{(n+1)}%
	         \ar[r]^-{  }
&  
Y_{(n)}%
	         \ar[r]^-{  }
&  
	\cdots	%
	         \ar[r]^-{  }
&  
Y_{(1)}%
	         \ar[r]^-{  }
&  
	Y_{(0)}	%
}
\]
in which the vertical maps become ever more highly connected as $n$ 
increases.  
The homotopy limit $Y$ is then both in $\R$ and 
weakly equivalent to $X$.
\end{proof}

The last of our three main theorems on resolving classes 
concerns their closure under certain extensions by cofibrations.

%
%

 \begin{thm}
 \label{thm:4}
 Let $\R$ be a strong resolving class with
 $ 
 \s\A^\wdg , 
 \s\B^\wdg\sseq \R$.
 Assume that $\A\smsh \A\sseq \A$, $\s\A\sseq \A$ 
 and that $   \A \smsh \s \B \sseq \s\B^\wdg $.
If $X$ sits in a cofiber sequence
\[
B \longrightarrow X\longrightarrow   A
\]
with $A\in \A^\wdg$ and $B\in \B^\wdg$, then $X\in \R$.
 \end{thm}
 

\begin{proof}
The proof  depends on a   decomposition
of the homotopy fiber of a principal cofibration:
if $P\to X\to \s Q$ is a cofiber sequence, then the 
suspension of the
homotopy fiber $F$ of $X\to \s Q$ is a half-smash product
\[
\mbox{$
\s F \simeq  
 \om \s Q
 \ltimes
 \s P 
\simeq 
\left(
 \bigvee_{n=0}^\infty Q^{\smsh n}
 \right)
\smsh
\s  P 
$}
\]
(see \cite[Prop. 4] {MR2029919}
for a proof; recall that 
  $Q^{\smsh 0} = S^0$).

Applying this to the cofiber sequence
$
\bigvee_1^k  \s B \to 
 \bigvee_1^k \s X \to 
 \bigvee_1^k \s A
$
we conclude that the homotopy fiber $F_k$ satisfies
\[
\mbox{$
\s F_k \simeq  
\left(
 \bigvee_{n=0}^\infty \left(\bigvee_1^k  A\right) ^{\smsh n} 
 \right)
 \smsh
 \s  B  .
 $}
\]
Therefore 
$\s   F_k\in  \A^\wdg \smsh \s\B^\wdg \sseq  \s\B^\wdg 
\sseq  \R
$, so
Theorem \ref{thm:2} implies that $F_k\in\R$.
Since 
  $\R$ is closed under extension by fibrations,
$\bigvee_1^k \s X\in \R$ for each $k$; 
  then Theorem \ref{thm:2} implies $X\in \R$.
\end{proof}

 Notice that, like the proof of Theorem \ref{thm:1}, 
this argument requires that we work with collections rather than 
individual spaces.  It is not enough to know that $B \in \R$;
rather, we need to know that a vast array of related spaces
are all in $\R$.

\subsection{Cone Length in Resolving Classes}
We finish this section by observing that Theorem \ref{thm:4}
implies a closure property for strong resolving classes
best expressed in terms of cone length.

Recall that throughout this section, the collections $\A$ and $\B$ 
are assumed to contain only simply-connected spaces.

\begin{thm}
\label{thm:DistanceClosure}
 Let $\R$ be a strong resolving class
 with
  $\s \A^\wdg, \s\B^\wdg\sseq \R$.
 Assume that $\A\smsh \A\sseq \A$, $\s\A\sseq \A$ 
 and that $   \A \smsh  \B \sseq \B^\wdg $.
 If there is a map $f: B\to K$ such that $B\in \B^\wdg$
 and $L_{\A^\wdg}(f) < \infty$, then $K\in \R$.
 \end{thm}
 
 \begin{proof}
%
 Write $\B_n$ for the collection of all spaces 
 $K$ such that there is a map $f: B\to K$ with  $B\in \B^\wdg$ and
 $L_{\A^\wdg}(f)\leq n$.    
 The hypotheses imply that 
 $\A\smsh  \s \B_n   \sseq  \s \B_n^\wdg $ for each $n$.
 We will prove that each $\B_n\sseq\R$ by induction on $n$.

 First of all,  $\B_0 = \B^\wdg \sseq \R$ by Theorem \ref{thm:2}.
Now suppose that $\B_n\sseq \R$, and let $K\in \B_{n+1}$.
The last step in an $\A^\wdg$-cone decomposition for $f: B\to K$ gives a 
cofiber sequence
\[
A_n\longrightarrow K_n \longrightarrow K \longrightarrow \s A_n
\]
with $K_n\in \B_n$ and $A_n\in \A^\wdg$.
 Therefore
   Theorem \ref{thm:4}
 implies that $K\in \R$, so $\B_{n+1}\sseq\R$, and the
 induction is complete.
 \end{proof}

 If we set $\A = \{ *\}$ in Theorem \ref{thm:DistanceClosure}
 (or even Theorem \ref{thm:4}), we recover Theorem \ref{thm:2}
 (which, of course, was used in the proof of Theorem \ref{thm:4}).  
  Taking $\B = \{ *\}$, on the other hand, 
 we derive the following corollary.
 
 \begin{cor}
 \label{cor:ConeLengthCorollary}
 If $\R$ is a strong resolving class with  $\s\A^\wdg \sseq \R$,
  $\A\smsh \A \sseq \A$ and $\s\A \sseq \A$, then 
 $\R$ contains every space $K$ with $\cl_{\A^\wdg}(K)< \infty$.
 \end{cor}

 
%


\section{Proof of Theorem \ref{thm:main}}

Now we apply the theory of resolving classes
to prove  Theorem \ref{thm:main}.
We begin  with two reductions.


  \begin{cor}
  \label{cor:10}
 Let $\R$ be a resolving kernel.
 \begin{problist}
 \item
   If  
  $\{ S^{nk+1}\st n \geq n_0\}\sseq \R$,  then 
 $\R$ contains all finite-type wedges of simply-connected 
 finite complexes.
 \item
 If $\R$ contains all simply-connected wedges of spheres, 
then $\R$ contains all  wedges of
simply-connected finite-dimensional spaces.
\end{problist}
 \end{cor}
 
 \begin{proof}
 We begin our proof of (a) by showing 
  that $\R$ contains all simply-connected 
 finite-type wedges of spheres.
 Since $\mathcal{A}= \{S^{nk}\st n \geq n_0\}$  
 is closed under smash product
 and  
 $\s\mathcal{A}\sseq \R$,    we may apply Theorem \ref{thm:1} 
 to conclude 
 $\s \mathcal{A}^\wdg \sseq \R$.
 Repeated application of Theorem \ref{thm:2} implies
 $\bigvee_1^m S^n\in \R$ for all $m\in \NN$ and all $n\geq 2$.
 Then   Theorems \ref{thm:1}  and  \ref{thm:2}
  give the result.
%
%
%
%

Now let $\mathcal{F}$ denote the collection of all simply-connected
finite complexes.  Since every space $K \in \s\mathcal{F}$ has 
finite cone length with respect to the collection of 
simply-connected finite-type wedges of spheres, 
Corollary \ref{cor:ConeLengthCorollary} 
implies that $\s\mathcal{F}\sseq \R$.
Since $\F\smsh\F\sseq \F$ and $\s \F\sseq \F$, 
Theorem \ref{thm:1} implies that $\s \F^\wdg \sseq \R$.
 Theorem \ref{thm:2} implies that $\F^\wdg\sseq\R$, 
proving (a).

For (b), 
observe that  the collection
of all simply-connected wedges of spheres
 is closed under suspension, smash
and finite-type wedge, 
so Corollary \ref{cor:ConeLengthCorollary}
implies that $\R$ contains every $2$-connected 
finite-dimensional space.
 Theorems \ref{thm:1} and \ref{thm:2}, applied to the collection
of all simply-connected finite-dimensional spaces, 
completes the proof.
\end{proof}

Next we establish a simple lemma characterizing the 
rational homotopy type of spaces  like the ones considered
in Theorem \ref{thm:main}.

\begin{lem}
\label{lem:RatTriv}
If $\map_*(X, S^n)\sim *$ for infinitely many values of $n$, 
then $\twdl H^*(X;\QQ) = 0$.
\end{lem}

\begin{proof}
Since $\s X$ also satisfies the hypotheses
and   $\twdl H^*(X;\QQ) = 0$ if and only if 
$\twdl H^*(\s X;\QQ) = 0$, we may assume 
that $X$ is simply-connected.
%
%
To show that $\twdl H^k(X;\QQ) = 0$, choose $n$ such that 
 $k\leq n-2$ and  $\map_*(X, S^n) \sim *$.  
The map
 $
 \ell_*: [X, \om^{n-k} S^{n}]
 \longrightarrow
  [X, \om^{n-k} S^{n}_\QQ]
 $
 induced by the rationalization of spaces
 $\ell : \om^{n-k} S^{n} \to \om^{n-k} S^{n}_\QQ$
 is rationalization of abelian 
 groups.
 But  
 $\map_*(X, S^{n})\sim *$ implies  $[X, \om^{n-k} S^{n}] = 0$, so 
 $[ X , \om^{n-k} S^n_{\QQ}] = 0$.
 Since every rational loop space splits as a product of
 Eilenberg-Mac Lane spaces, 
 $K(\QQ, k )$ is a retract of $\om^{n-k} S^{n}_\QQ$
and
 $
 \twdl H^k(X; \QQ) =  [ X, K(\QQ, k)] = 0
 $.
  %
%
%
\end{proof}

Now we are prepared to prove   
Theorem \ref{thm:main}.

\begin{proof}[Proof of Theorem \ref{thm:main}]
  Let $X$ be a space satisfying the hypotheses
of Theorem \ref{thm:main}, and suppose that $K$
is a simply-connected finite-dimensional 
CW complex.  Since   the resolving kernel
\[
\R = \{ K  \st \map_*(X, K) \sim *\}
\]
  contains $S^{nk+1}$ for all sufficiently large $n$,  
  Corollary \ref{cor:10}(a) guarantees that it contains
  all finite-type wedges of simply-connected finite
  complexes. 
  According to Corollary \ref{cor:10}(b), it suffices to show
 that 
 \[
 [ \s^t X, W]  =  \pi_t( \map_*(X, W)) = 0
 \]
 for all $t\geq 0$,
 where $W = \bigvee_{i\in \I} S^{n_i}$
 is a simply-connected wedge of spheres.
   So we choose  a typical map 
   $f: \s^t X \to W$ and attempt to show it is trivial.

Since  $X$ has finite type, each skeleton
 $(\s^t X)_k$ is compact so 
$f((\s^t X)_k)$ is contained in a finite subwedge $V \sseq W$.
It follows that $f$ factors through the inclusion of a \textit{countable}
subwedge of $W$, and so we may as well assume in retrospect
that $W$ is itself a countable wedge.
Furthermore, we know from Corollary \ref{cor:10}(a) that
$\map_*(X, V) \sim *$, so there is a
homotopy commutative diagram 
\[
\xymatrix{
&&&& (\s^t X)_k
\ar[d]
\ar@(r,u)[rrdd]^{f|_{(\s^t X)_k}}
\ar@(l,ru)[lllldd]
\\
&&&& \s^t X \ar[d]^{*} \ar@/_/[lld]_f
\\
V
\ar[rr]^-i
\ar@(rd,ld)[rrrr]_-{\id_V}
&&W \ar[rr]^-q && V \ar[rr]^-{i} && W, 
}
\]
in which $q$ is the collapse map to $V$ and $i$ is the inclusion.
This shows that $f|_{(\s^t X)_k}\simeq *$ for every $k$, and hence that 
$f$ is a phantom map.  
The   conclusion $f\simeq *$   follows from
Lemma \ref{lem:RatTriv} 
and
  Proposition \ref{prop:BLA}.  
  Thus $\map_*(X,K) \sim *$ if $K$ is simply-connected.

Finally allow the possibility that $K$ is not simply-connected 
and assume that 
$\pi_1(X)$ has no nontrivial perfect quotients.
Write $G = \im ( \pi_1(f))$
and consider the covering  $q: L \to K$ 
corresponding to the subgroup $G \sseq \pi_1(K)$.
There is a lift $\phi$ in the diagram
\[
\xymatrix{
&& L \ar[d]^q
\\
\s^t X \ar[rr]^-f\ar@/^/@{..>}[rru]^\phi  && K
}
\]
which induces a surjection on 
fundamental groups.  If $G = \{ 1\}$, then $L$
is simply-connected and finite-dimensional, 
and $\phi\simeq *$ by the simply-connected 
part of Theorem \ref{thm:main}.  
If $G \neq \{ 1\}$, then it is not perfect, so there
is a nontrivial map $u: L \to K(A,1)$ 
for some abelian group $A$
 ($u$ can be chosen so that 
$u_*: \pi_1(L) \to A$ is 
 abelianization).
Thus $\phi$ is nonzero on cohomology and so
its suspension
 $\s \phi : \s^{t+1} X\to \s L$ is also nontrivial, 
 contradicting the
   simply-connected part of  Theorem \ref{thm:main}.

%
%
\end{proof}

\section{Discussion}

\subsection{Some Comments on Theorems}
%
Corollary \ref{cor:ConeLengthCorollary} 
  implies a bit more than is actually stated.
Since  the collection
$\cl_{<\infty}(\A^\wdg)$ of spaces $K$ with finite $\A^\wdg$-cone length
is closed under suspension and smash, 
we   find that 
\[  
\cl_{<\infty}(\A^\wdg)\sseq 
\cl_{<\infty}( (\cl_{<\infty}(\A^\wdg))^\wdg )
\sseq
\cl_{<\infty}(\cl_{<\infty}( (\cl_{<\infty}(\A^\wdg))^\wdg )^\wdg )
\sseq
\cdots
\sseq \R
.
\]
These are genuine improvements:  for example, they 
imply that for $X$ as in Theorem \ref{thm:main}
\[
\map_* 
\left( 
\mbox{$X, \bigvee_{n=2}^\infty ( \om S^{n+1})_{n^2} $}
\right)
\sim *
\]
(the subscript $n^2$ denotes dimension of a CW skeleton);
 in this example, the
 target 
 has infinite Lusternik-Schnirelmann category and hence infinite
cone length with respect to any collection $\A$.

Perhaps the reader is thinking that the insistence on \textit{finite-type}
wedges in Theorems \ref{thm:1} and \ref{thm:2}
is simply a matter of expediency---that    
`finite-type'   could be eliminated
 if we tried hard enough.
But it is not true that a space $X$ satisfying the conditions of 
Theorem \ref{thm:main} satisfies the condition $\map_*(X , W)\sim *$ 
for \textit{all} wedges of finite complexes $W$; 
indeed all such spaces $X$ that are not killed by suspension
 \textit{must} have  nontrivial 
maps into certain wedges of finite complexes.

\begin{thm}
\label{thm:ExistPhantom}
 If $X$ is as in Theorem \ref{thm:main} and $\s X \not\sim *$, 
then the universal phantom map 
$\Theta_X: X\to \bigvee_{n=1}^\infty\s X_n$
is nontrivial.
\end{thm}

\begin{proof}
Just as in the proof of Theorem \ref{thm:main}, we can show 
that every map from $\s X$ to a wedge of finite complexes must be a 
phantom map.  It is shown in \cite[Thm. 2]{MR1217076} 
that if $\Theta_X \simeq *$
then $\s X$ is a retract (up to homotopy) of a wedge of finite complexes.  
Thus $\Theta_X \simeq *$ simultaneously implies   that
 $\id_{\s X}$ is   a phantom map and   that $\s X$ is not the 
 domain of any nontrivial phantom maps.
 \end{proof}
 
 The conclusion $\map_*(X, K) \sim *$ for non simply-connected 
 spaces $K$ in Theorem \ref{thm:main}
 can be deduced more generally.  It is valid provided 
    no homomorphism from 
 $\pi_1(X)$ to $\pi_1(K)$ can have a nontrivial perfect group
 as its image.  This is the case, for example, if $\pi_1(K)$ is 
 hypoabelian, 
 regardless of the structure of $\pi_1(X)$.
  The restriction on fundamental groups cannot be entirely dispensed
  with, however.
  Any nontrivial acyclic $2$-dimensional complex $X$ 
  satisfies $\map_*(X, K)\sim *$
 for all finite-dimensional spaces $K$ with hypoabelian 
fundamental groups, but $\map_*(X, X) \not\sim *$.
  Such spaces also demonstrate the
need for the hypothesis $\s X \not \sim *$ in 
Theorem \ref{thm:ExistPhantom}.



\subsection{The Sphere Code of a Space}
Corollary \ref{cor:10}(a) suggests an interesting array of questions.
Define the \term{sphere code} of a space $X$ to 
be the set 
\[
 \sigma(X)= \{ n\in \NN  \st \map_*(X, S^n) \sim *\}.
 \]
 (This can be extended to resolving classes:   
 $\sigma(\R) = \{ n \st S^n \in \R\}$.)
We have shown in Corollary \ref{cor:10}(a) 
 that if $\sigma(X)$ contains an infinite
arithmetic sequence of the form
$\{ nk + 1 \st n \geq n_0\}$, then
$\sigma(X) = \NN$.     What else can 
be said of it?  

We offer only a few simple observations, followed by some 
questions.

\begin{prop}
Let $X$ and $Y$ be CW complexes.
\begin{problist}
\item
If $2\in \sigma(X)$, then $1\in \sigma(X)$;  
if $4\in \sigma(X)$, then $3\in \sigma(X)$;  
if $8\in \sigma(X)$, then $7\in \sigma(X)$.
\item
If $X$ is $p$-local ($p$ is an odd prime) and $2n\in \sigma(X)$, then
$2n-1\in \sigma(X)$.\footnote{Thus if 
$\{ nk+1+\epsilon_n \st n\geq n_0, \epsilon_k \in \{ 0,1\}
\}\sseq \sigma(X)$, then $\sigma(X) = \NN$.}
 \item
$\sigma(X\wdg Y) = \sigma(X) \cap \sigma(Y)$.
\item
$\sigma(X\smsh Y) \supseteq \sigma(X) \cup \sigma(Y)$.
\end{problist}
\end{prop}

We omit the proof and  offer a few questions about 
sphere codes.
\begin{enumerate}
\item
If $\sigma(X)\neq \NN$,  can $\sigma(X)$ contain an 
infinite arithmetic sequence?
\item
Can $\sigma(X)$ be infinite without being all of $\NN$?
 \item
Is it possible to classify the sphere codes of spaces?
Is there a space $X$ such that $\sigma(X) = \{ 1\ \mbox{and all primes}\}$? 
  \item
There is a closure operation for subsets $N\sseq\NN$
given by $\overline N = \sigma( \R)$, where $\R = \Theta(\{ S^n\st n\in N\})$; 
can it be described numerically?
%
\end{enumerate}

%
%
%
%
%
%


\section{Appendix:  Reduction from Algebra to Topology}

The following theorem gives the
 basic algebraic input for the proof of the Sullivan conjecture
 (see Section \ref{subsection:U} for notation and terminology).
 
 \begin{thm}[Miller, Carlsson]
 \label{thm:SullivanAlgebra}
 The unstable $\A_p$-module $\twdl H^*(B\ZZ/p; \FF_p)$ is 
 reduced    and $\twdl H^*(B\ZZ/p; \FF_p) \otimes J(n)$ is   
  injective   for all $n\geq 0$.
 \end{thm}
 
We will not prove this here.\footnote{A proof can be found
 in \cite[Lem. 2.6.5 \& Thm. 3.1.1]{MR1282727}.}
Rather, we show how these  algebraic properties 
 guarantee that $\map_*(B\ZZ/p,S^{2n+1})\sim *$ for $n\geq 1$.
 We begin by reviewing some preliminary material on
the category $\U$ of unstable $\A_p$-algebras.  Then we 
give a brief account of   
  Massey-Peterson towers and finally
  derive
from Theorem \ref{thm:SullivanAlgebra}
  that $\map_*(B\ZZ/p, S^{2n+1})\sim *$ for all 
$n$.

\subsection{Unstable Modules over the Steenrod Algebra}

\label{subsection:U}

The cohomology functor $H^*(\?; \FF_p)$
takes its values in the category $\U$ of   unstable
modules and their homomorphisms.  
An \term{unstable module} over the Steenrod algebra $\A_p$
is a graded $\A_p$-module $M$ satisfying   $P^I(x) = 0$ if 
$e(I) > |x|$, where $e( I )$ is
the excess of $I$ and $|x|$ is the degree
of $x\in M$.     
%
We begin with some basic algebra of unstable modules, 
all of which is (at least implicitly) in \cite{MR1282727}.

\medskip

\paragraph{\bf Suspension of Modules}
An unstable module $M\in \U$ has a \term{suspension}
 $\s M\in \U$
given by $(\s M)^n  = M^{n-1}$.  The functor $\s : \U \to \U$
has a left adjoint $\om$ and a 
right adjoint $\twdl \s$.\footnote{$\twdl \s M$
is the largest suspension module contained in $M$.}
  A module
$M$  
is called \term{reduced} if $\twdl \s M = 0$.

\medskip

\paragraph{\bf Projective and Injective Unstable Modules}
In the category $\U$, there are free modules
$F(n) = \A_p/ E(n)$, where $E(n)$ is the smallest left
ideal containing all Steenrod powers $P^I$
with excess $e(I) > n$.   It is easy to see that 
the assignment $f\mapsto f([1])$ defines 
natural isomorphisms
\[
\Hom_\U (F(n) , M) \xrightarrow{\, \cong\, } M^n .
\] 
This property defines $F(n)$ up to 
natural isomorphism, and shows that $F(n)$ deserves to 
be called a \term{free} module on a single generator of
dimension $n$.  More generally, the free module on 
a set $X= \{ x_\alpha\}$ with $|x_\alpha| = n_\alpha$
is (up to isomorphism) the sum $\bigoplus F(n_\alpha)$
(see \cite[\S 1.6]{MR1282727} for details).
 
 A graded $\FF_p$-vector space $M$ 
is   of \term{finite type} if  
  $\dim_{\FF_p}( M^k) < \infty$  for each $k$.
  Since $\A_p$ is of finite type, so is $F(n)$.
 
 The functor   which takes $M\in \U$ and returns 
the dual $\FF_p$-vector space $(M^n)^*$ is representable:
there is a module $J(n) \in \U$ and a natural isomorphism
\[
\Hom_\U ( M, J(n) ) \xrightarrow{\, \cong\, }
\Hom_{\FF_p}( M^n,\FF_p) .
\]
Since finite sums of vector spaces are also finite products, 
these functors are exact, so the module $J(n)$ is an injective
object in $\U$.

 \subsection{The Functor $\overline \tau$}
 
 In \cite[Thm 3.2.1]{MR1282727} it is shown that 
 for any module $H \in \U$, the functor $H \otimes_{\A_p} \?$
 has a left adjoint, denoted $(? : H)_\U$.    
Fix a module $H$ (to  stand in for $\twdl H^*(X; \FF_p)$)
 and write $\overline \tau$ for the
 functor $(\? :H)_\U$; this is intended to evoke the 
 standard notation $\overline T$ for the special case $H = \twdl H^*(
 B\ZZ/p; \FF_p)$.

\begin{lem}
\label{lem:main}
Let $H\in \U$ be a reduced unstable
module of finite type
and suppose that $H\otimes J(n)$ is injective in $\U$
for every $n\geq 0$.
Then
\begin{problist}
\item
$\overline \tau$ is exact,
\item
$\overline \tau$ commutes with suspension,
\item
if $M$ is free and of finite type, then so is $\overline \tau(M)$, and 
\item
if $H^0 = 0$, then  $\overline \tau(M) = 0$
for any finite module $M\in \U$.
\item
if $H^0 = 0$, then 
$\Ext_{\U}^{s}(H,\s^{s+t}M)=0$ for all $s,t\geq 0$.
\end{problist}
\end{lem}

\begin{proof}
These results are covered in 
Sections 3.2 and 
 3.3 of \cite{MR1282727}.
 Specifically, parts (a) and (b) are proved as in
 \cite[Thm. 3.2.2 \& Prop. 3.3.4]{MR1282727}.  
 Parts (c) and (d) may be proved following
  \cite[Lem. 3.3.1 \& Prop. 3.3.6]{MR1282727}, but since there are
  some changes needed, we prove those parts here.
%
%

Write $d_k  = \dim_{\FF_p}( H^k)$;
then there are natural isomorphisms
\begin{eqnarray*}
\Hom_\U \left( 
\overline \tau( F(n)) ,   M \right
) 
&\cong&
\Hom_\U \left( 
  F(n) , H \otimes M \right
) 
\\
&\cong &
\Hom_\U \left( \mbox{$\bigoplus_{i+j = n} 
F(i)^{\oplus d_{j}}
,  M$} 
\right) , 
\end{eqnarray*}
proving (c) in the case of a free module on one generator.  
Since $\overline \tau$ is a left adjoint, it
commutes with colimits (and sums in particular), 
we derive the full statement of (c).

If $H^0 = 0$, then $d_0 = 0$ and   $\overline \tau(F(n))$
is a sum of free modules $F(k)$ with $k < n$. 
Since $F(0) = \FF_p$, we see  that  $\overline \tau(\FF_p) = 0$;
then  (a), together with the fact that $\overline \tau$ commutes
with colimits,  implies that $\overline \tau (M) = 0$ for all trivial modules 
 $M$.    Finally, any finite module $M$ has filtration all of whose
 subquotients are trivial, and (d) follows.
%
%

To prove (e), let  $P_*\to M\to 0$ is a free resolution of $M$ in $\U$.
Parts (a), (c) and (d) together
 imply that 
$\overline \tau(P_*) \to 0 \to 0$ is   a free resolution of $0$, so  
\begin{eqnarray*}
\Ext_\U^s \left( M, \s^{s+t} H \right)
&=&
\Ext_\U^s \left( M, H \otimes \s^{s+t}\FF_p)\right)
\\
&=&
H^s\left(  \Hom\left( P_*, H \otimes \s^{s+t}\FF_p\right)\right)
\\
&\cong &
H^s\left( \Hom \left( \overline \tau( P_*), \s^{s+t} \FF_p \right)\right)
\\
&=&
\Ext_\U^s \left( 0, \s^{s+t}\FF_p\right)
\\
&=& 0.
\end{eqnarray*}
\end{proof}

 \subsection{Massey-Peterson Towers}

Cohomology of spaces has more structure than just 
that of an unstable $\A_p$-module:  it has a cup product
which makes $H^*(X; \FF_p)$ into an 
\term{unstable algebra} over $\A_p$.  The category 
of unstable algebras is denoted $\K$.

The forgetful functor $\K\to \U$ 
has a left adjoint $U: \U \to \K$.  A space $X$  is said to have
\term{very nice} cohomology if $H^*(X) \cong U(M)$ for some
unstable module $M$ of finite type.

Since $U(F(n)) \cong H^*(K(\ZZ/p, n))$, 
there is a contravariant functor $K$ which carries a free module
$F$ to a generalized Eilenberg-Mac Lane space (usually abbreviated
GEM) $K(F)$ such that 
$H^*(K(F)) \cong U(F)$.   If $F$ is free, then so is $\om F$, and 
  $K(\om F) \simeq \om K(F)$.

\begin{lem}
\label{lem:Maps2KF}
For any $X$, 
$[ X, K(F)] \cong \Hom_\U( F, \twdl H^*(X))$.
\end{lem}

It is shown in \cite{MR764593,MR546361,MR0226637} 
that if $H^*(Y) \cong U(M) $   and $P_*\to M\to 0$
is a free resolution in $\U$, 
then $Y$ has a \term{Massey-Peterson tower}
\[
\xymatrix{
\cdots \ar[r] &
Y_{s}\ar[r] \ar[d]& 
Y_{s-1} \ar[r] \ar[d] &
\cdots \ar[r] &
Y_1 \ar[r] \ar[d]&
Y_0\ar[d]
\\
& K(\om^s P_{s+1}) & K(\om^{s-1} P_s) && K(\om P_2) & K(P_1)
}
\]
in which 
\begin{enumerate}
\item
$Y_0 = K(P_0)$, 
\item 
each homotopy group $\pi_k(Y_s)$ is a finite $p$-group, 
\item
the limit of the tower is the $p$-completion $Y^\smsh_p$, 
\item
each sequence $Y_s \to Y_{s-1}\to K(\om^{s-1} P_s)$
is a fiber sequence, and
\item
the compositions   
$
\om K(\om^{s-1} P_s) 
\to
Y_s
\to
 K(\om^s P_{s+1})
$
 can be naturally
 identified with $K( \om^s d_{s+1} )$, where 
$d_{s+1}: P_{s+1}\to P_s$ is the differential
 in the given free resolution.
\end{enumerate}

We use Massey-Peterson towers to 
give a criterion for the vanishing of homotopy sets.


\begin{thm}
\label{thm:FakeHarper}
Suppose  $Y$ is a simply-connected CW complex with
$H^*(Y) = U(M)$ for some finite $M\in \U$
and $Z$ is a  CW complex
 of finite type with $\twdl H^*(Z; \ZZ[{1\over p}]) = 0$.
If    
$
\Ext_\U^s (M , \s^{s} \twdl H^*(Z)) =0 
$
for all $s\geq 0$, then $[   Z, Y ] = *$.
\end{thm}

\begin{proof}
According to \cite[Thm. 4.2]{MR764593}, the 
natural map $Y \to Y^\smsh_p$ induces 
a bijection $[ Z, Y]  \xrightarrow{\cong}
[Z, Y^\smsh_p]$, so it suffices
to show   $[Z , Y^\smsh_p ] = *$.
Since $H^*(Y) = U(M)$, $Y$ has a Massey-Peterson 
tower, whose homotopy limit is $Y^\smsh_p$.
  Let $f_s$ be the composite $Z\to Y \to Y_s$; we will 
show by induction that $f_s \simeq *$ for all $s$.

Since $Y_0$ is a GEM, 
$f_0$ is determined by its effect on cohomology;
and since
$
\Hom_\U ( M, \twdl H^*(Z)) = 
\Ext_\U^0 (M , \s^{0} \twdl H^*(Z)) = 0
$,
  $f_0$ is trivial on cohomology, and hence trivial.
 Inductively,
suppose $f_{s-1}$
is trivial.  We have the following situation
\[
\xymatrix{
K ( \om ^s P_{s-1}) \ar@(ru,lu)[rr]^-{K(\om^s d_{s})}
\ar[r]
&
\om Y_{s-1}\ar[r] 
\ar@(rd,ld)[rr]_-{*}
&
K( \om^s P_s) \ar[r] \ar@(ru,lu)[rr]^-{K(\om^s d_{s+1})}
&Y_s\ar[d]\ar[r] 
& K(\om^s P_{s+1} )
\\
&&&
Y_{s-1} .
}
\]
Now apply  $[ Z, \? ]$ to this diagram and observe that
Lemma \ref{lem:Maps2KF}, together with 
the  isomorphism 
$\Hom_\U( \om^s P, H) \cong \Hom_\U( P, \s^s H)$
(with $H =\twdl  H^*( Z)$),
gives
\[
\xymatrix{
&& [Z, Y_{s+1}]\ar[d]
\\
\Hom_\U( 
P_{s-1}, \s^s H) 
\ar[r]^-{d_{s}^*}
\ar@(rd,ld)[rr]_{*}
&
\Hom_\U (
P_s, \s^s H ) \ar[r]^-\alpha \ar@(ru,lu)[rr]^(.7){ d_{s+1}^*}  
&[Z,Y_s]\ar[d]\ar[r]^-\beta
& \Hom_\U( 
P_{s+1} , \s^s H)
\\
&&
[Z, Y_{s-1}] .
}
\]
Exactness at $[Z, Y_s]$ implies that the homotopy class
$[f_s]$ is equal to $\alpha( g_s)$ for some 
$g_s\in \Hom_\U (P_s, \s^s H )$.
Since $[f_s]$ is in the image of the vertical map from 
$[Z,Y_{s+1}]$, it is in the kernel of $\beta$, so 
$d_{s+1}^*(g_s) = \beta( [f_s]) = 0$; in other words, 
$g_s$ is a cycle representing an element   
$[g_s] \in \Ext_\U^s(M, \s^s H)$.   Since
$ \Ext_\U^s(M, \s^s H) = 0$, we 
conclude that there is
a 
$g_{s-1}\in 
\Hom_\U(P_{s-1}, \s^s H)$
such that 
 $g_s = d_s^*(g_{s-1})$.  Therefore 
$
[f_s] = \alpha ( d_s^*(g_{s-1}) ) =  [ *]
$.

Since every map $f: Z\to Y$
is trivial on composition to $Y_s$ for each $s$, the 
  exact sequence
$
* \to \limone   [ Z, \om Y_s]
\to 
[  Z, Y^\smsh_p ] 
\to
 \lim [  Z, Y_s] \to *
$
reduces to a surjection 
$
\limone   [ Z, \om Y_s]
\to
[  Z, Y^\smsh_p ]
$, 
and Lemma \ref{lem:CompactLimOneVanish}
finishes the proof.
%
\end{proof}

\subsection{Maps from $B\ZZ/p$ to Odd Spheres}

We are finally able to establish Theorem \ref{thm:blarg} which, by virtue of 
Theorem \ref{thm:SullivanAlgebra},
 implies the weak contractibility
 of the space of maps from $B\ZZ/p$ to   odd spheres.

\begin{proof}[Proof of Theorem \ref{thm:blarg}]
 Write $H = \twdl H^*(X;\FF_p)$; thus
   $H\in \U$ is a reduced module of finite type 
and   $H\otimes J(n)$ is injective for all $n$.
Since $H^*(S^{2n+1}) = U( \s^{2n+1}\FF_p)$, 
  the result follows from applying Lemma \ref{lem:main}(e) and 
Theorem \ref{thm:FakeHarper} to the spaces $Z = \s^t X$ for 
$t\geq 0$.
\end{proof}

%
 

\begin{cor}
\label{cor:SulCon}
$\map_*(B\ZZ/p, S^{2n+1})\sim *$ for all $n\geq 1$.
\end{cor}

\begin{proof}
Lemma \ref{lem:RatTriv} and 
Theorem \ref{thm:SullivanAlgebra}
imply that we may take $X = B\ZZ/p$ in Theorem \ref{thm:blarg}.
\end{proof}

\end{document}